\documentclass{amsart}
\usepackage{amsfonts,amsmath,amssymb}
\usepackage{mathrsfs,mathtools}
\usepackage{enumerate}
\usepackage{xspace}
\usepackage{hyperref}
\DeclareMathAlphabet{\mathpzc}{OT1}{pzc}{m}{it}

\newcommand{\Nin}{\,{\mbox{\,\raisebox{5.0pt} {\tiny$\circ$} \kern-9.6pt}\N }}
\newcommand{\Ninn}{{\mbox{\,\raisebox{4.1pt} {\tiny$\circ$} \kern-8.1pt}\N }}
\newcommand{\R}{\mathbb{R}}
\newcommand{\Y}{\mathpzc{Y}}
\newcommand{\N}{\mathpzc{N}}
\newcommand{\C}{\mathcal{C}}
\newcommand{\V}{\mathbb{V}}
\newcommand{\T}{\mathscr{T}}

\newcommand{\Ss}{\mathscr{S}}
\newcommand{\HL}{ \mbox{ \raisebox{6.7pt} {\tiny$\circ$} \kern-8.9pt} {H_L^1} }
\newcommand{\HLn}{{\mbox{\,\raisebox{4.7pt} {\tiny$\circ$} \kern-8.5pt}{H}^1_L  }}
\newcommand{\Tr}{\mathbb{T}}
\DeclareMathOperator*{\tr}{tr_\Omega}

\newcommand{\Laps}{(-\Delta)^s}
\newcommand{\GRAD}{\nabla}
\newcommand{\DIV}{\textrm{div}}
\newcommand{\diff}{\, \mbox{\rm d}}
\newcommand{\ie}{i.e.,\@\xspace}
\newcommand{\Hs}{\mathbb{H}^s(\Omega)}

\newcommand{\Hsd}{\mathbb{H}^{-s}(\Omega)}
\newcommand{\ue}{\mathscr{U}}
\newcommand{\ve}{\mathscr{V}}
\newcommand{\ze}{\mathscr{Z}}
\newcommand{\calI}{{\mathcal I}}

\numberwithin{equation}{section}
\newtheorem{theorem}[equation]{Theorem}
\newtheorem{lemma}[equation]{Lemma}
\newtheorem{proposition}[equation]{Proposition}

\theoremstyle{definition}

\theoremstyle{definition}
\newtheorem{remark}[equation]{Remark}

\begin{document}
\title[The classical, thin and fractional elliptic obstacle problems]{Convergence rates for the classical, thin and fractional elliptic obstacle problems}

\author[R.H.~Nochetto]{Ricardo H.~Nochetto}
\address[R.H.~Nochetto]{Department of Mathematics and Institute for Physical Science and Technology,
University of Maryland, College Park, MD 20742, USA.}
\email{rhn@math.umd.edu}
\thanks{RHN has been supported in part by NSF grants DMS-1109325 and DMS-1411808.}

\author[E.~Ot\'arola]{Enrique Ot\'arola}
\address[E.~Ot\'arola]{Department of Mathematics, University of Maryland, College Park, MD 20742, USA and Department of Mathematical Sciences, George Mason University, Fairfax, VA 22030, USA.}
\email{kike@math.umd.edu}
\thanks{EO has been supported in part by NSF grants DMS-1109325 and DMS-1411808}

\author[A.J.~Salgado]{Abner J.~Salgado}
\address[A.J.~Salgado]{Department of Mathematics, University of Tennessee, Knoxville, TN 37996, USA.}
\email{asalgad1@utk.edu}
\thanks{AJS has been supported in part by NSF grant DMS-1418784.}

\subjclass[2000]{35J70,           
35R11,                    
35R35,                    
41A29,                    
65K15,                    
65N15,                    
65N30}                    

\keywords{obstacle problem, thin obstacles, free boundaries, finite elements, fractional diffusion, anisotropic elements.}

\begin{abstract}
We review the finite element approximation of the classical obstacle problem in energy and max-norms and derive error estimates for both the solution and the free boundary. On the basis of recent regularity results we present an optimal error analysis for the thin obstacle problem. Finally, we discuss the localization of the obstacle problem for the fractional Laplacian and prove quasi-optimal convergence rates.
\end{abstract}

\date{Submitted on \today.}

\maketitle

\section{Introduction}
\label{sec:introduccion}
Obstacle and obstacle-type problems serve as a model for many phenomena and stand as the prototype for many theories: variational inequalities, constrained minimization and free boundary problems, to name a few. They are also the first truly nonlinear and nonsmooth problem one encounters in the study of elliptic partial differential equations (PDE) and their numerical approximation: linearization fails!

The purpose of this short note is to review the finite element approximation of the solution to the classical, thin and fractional obstacle problems --- the three basic prototypes. Some of the results in this note are classical but are scattered in the literature. 
Others, specially those about nonlocal operators, are completely new.

\section{The classical obstacle problem}
\label{sec:classical_obstacle}
Let $\Omega$ be an open, bounded and connected domain of $\R^n$
with boundary $\partial \Omega$.
Let $f \in H^{-1}(\Omega)$ and $\psi \in H^1(\Omega) \cap C(\bar{\Omega})$ satisfying $\psi \leq 0$
on $\partial \Omega$. We consider
the following classical elliptic obstacle problem:
\begin{equation}
\label{eq:class_obstacle}
 u \in \mathcal{K}: \qquad (\nabla u , \nabla(u-v)) \leq \langle f,u-v \rangle \qquad \forall v \in \mathcal{K},
\end{equation}
where $\mathcal{K}$ denotes the convex set of admissible displacements
\begin{equation}
\label{K_classical}
\mathcal{K}:= \{v \in H_0^1(\Omega): v \geq \psi \textrm{ a.e. in } \Omega   \},
\end{equation}
$(\cdot,\cdot)$ corresponds to the inner product in $L^2(\Omega)$ and $\langle \cdot,\cdot \rangle$
denotes the duality pairing between $H^{-1}(\Omega)$ and $H_0^{1}(\Omega)$. It is known that under
these conditions, \eqref{eq:class_obstacle} has a unique solution $u \in H_0^1(\Omega)$ and there exists
a nonnegative Radon measure $\mu$ such that
\[
 \int_{\Omega} v \diff \mu = (\nabla u, \nabla v) - \langle f,v \rangle \qquad \forall v \in H_0^1(\Omega),
\]
and 
$
 \textrm{supp}(\mu) \subset \Lambda := \{ x \in \Omega: u(x) = \psi(x) \},
$
where $\Lambda$ is the so-called \emph{coincidence} or \emph{contact set}; see 
\cite{Friedman,KS,PSU:12} for details. In addition, if 
$f \in L^2(\Omega)$ and $\psi \in H^2(\Omega)$, then $u \in H^2(\Omega)$ provided $\partial \Omega$
is $C^{1,1}$ or $\Omega$ is a convex polyhedron. In this case $\mu$ is absolutely continuous with respect to Lebesgue measure 
$\diff \mu = \lambda \diff x$ with $0 \leq \lambda \in L^2(\Omega)$. With these assumptions we can rewrite \eqref{eq:class_obstacle}
as a {\it complementarity system} \cite{Friedman,KS,PSU:12}:
\begin{equation}
\label{eq:complementary_system}
\lambda = -\Delta u - f \geq 0, \quad u - \psi \geq 0, \quad (\Delta u + f)(u - \psi) = 0 \quad \textrm{a.e. in }\Omega.
 \end{equation}

We define the \emph{non-coincidence set}
$
  \Omega^+ := \{ x \in \Omega: u(x) > \psi(x)  \}
$
which, if $u$ is continuous, is an open set, and the \emph{free boundary}
$
\Gamma := \partial \Omega^+ \cap \Omega. 
$

We remark that for the obstacle problem \eqref{eq:class_obstacle}, the correspondence between the right hand side
$f$ and $u$ is non-linear, non-differentiable, and most notably not one-to-one since a change of $f$ within the contact set
$\Lambda$ may yield no change in $u$.

\subsection{Finite element approximation}
\label{subsec:fe_approximation}

Let us now describe the discretization of problem \eqref{eq:class_obstacle}.
To avoid technical difficulties let us assume that the boundary of $\Omega$ is polyhedral.
Let $\T = \{ T \} $  be a mesh of $\Omega$ into cells $T$ 
such that
\[
  \bar \Omega = \bigcup_{T \in \T} T, \qquad
  |\Omega| = \sum_{T \in \T} |T|,
\]
where $T \subset \R^n$ is a cell that is isoparametrically equivalent either to the unit cube 
$[0,1]^n$ or the unit simplex in $\R^n$.
The partition $\T$ is assumed to be conforming or compatible, \ie the 
intersection of any two cells $T$ and $T'$ in $\T$ is either empty or a common lower
collection of all conforming meshes, which are refinements of a
  common mesh $\T_0$. We also assume 
$\Tr$ to be \emph{shape regular}, \ie there exists a constant $\sigma > 1$ such that
\begin{equation}
\label{shape_reg}
 \max \left\{ \sigma_T : T \in \T, \T \in \Tr \right\} \leq \sigma,
\end{equation}
where $\sigma_T:=h_T/\rho_T$ is the shape coefficient of $T$. For simplicial elements,
$h_T = \textrm{diam}(T)$ and $\rho_T$ is the diameter of the largest sphere inscribed in $T$; see 
\cite{CiarletBook}.
For the definition of $h_T$ and $\rho_T$ in the case of cubes, we refer to~\cite{CiarletBook}.
The collection of meshes $\Tr$ is \emph{quasiuniform} if for all $\T \in \Tr$ and all $T,T' \in \T$ we have $h_T \lesssim h_{T'} \lesssim h_T$. 
In this case the \emph{mesh size} is $h_\T := \max\{ h_T : T \in \T \}$.

Given a mesh $\T \in \Tr$, we define the following finite element space 
\begin{equation}
\label{eq:fe_space}
  \V(\T) := \left\{
            W \in \C^0( \bar{\Omega}): \ W|_T \in \mathcal{P}(T) \ \forall T \in \T, \
            W|_{\partial \Omega} = 0
          \right\} \subset H_0^1(\Omega),
\end{equation}
where if $T$ is a simplex $\mathcal{P}(T)=\mathbb{P}_1(T)$, \ie polynomials of degree at most one, and if $T$ is a cube, $\mathcal{P}(T)=\mathbb{Q}_1(T)$, that is, polynomials of degree at most one in each variable.

Given a mesh $\T \in \Tr$ and $T\in \T$, we denote by $\N(T)$
the set of nodes of $T$. We then define $\N(\T):= \cup_{T \in \T} \N(T)$ and 
$~\Nin(\T):= \N(\T) \cap \Omega$. Any discrete function $W \in \V(\T)$ is characterized
by its nodal values on the set $\Nin(\T)$:~
$
 W = \sum_{z \in~\Ninn(\T)} W(z) \phi_{z}
$,
where $\{\phi_{z}\}_{z \in ~\Ninn(\T)} $ is the canonical basis of $\V(\T)$,
that is $\phi_z(\tilde z) := \delta_{z\tilde z}$ for $z,\tilde z \in \Nin(\T)$.

The discrete admissible set $\mathcal{K}_{\T}$, \ie the discrete counterpart of $\mathcal{K}$, is
\begin{equation}
\label{eq:K_discrete}
\mathcal{K}_{\T}:= \{ W \in \V(\T): W \geq I_{\T} \psi \textrm{ on } \Omega \},
\end{equation}
where $I_{\T}$ is the Lagrange interpolation operator $I_{\T}\psi =  \sum_{z \in \Ninn(\T)}  \psi(z) \phi_{z}$.
The set $\mathcal{K}_{\T}$ is nonempty, convex, closed but in general is not a subset of $\mathcal{K}$. The classical literature would refer then to this setting 
as a nonconforming approximation.
The \emph{discrete problem} reads
\begin{equation}
\label{eq:discrete_obstacle}
 U_{\T} \in \mathcal{K}_{\T}: \qquad (\nabla U_{\T} , \nabla(U_{\T}-V)) \leq ( f,U_{\T}-V ) \qquad \forall V 
 \in \mathcal{K}_{\T}.
\end{equation}
It is well known that the solution $U_{\T} \in \V(\T)$ exists and is unique \cite{BHR:77,CiarletBook}.

\begin{remark}[positivity preserving operator]
\label{rm:CN:00}
In \eqref{eq:K_discrete} the Lagrange interpolant can be replaced 
by the positivity preserving operator $\Pi_{\T}:L^1(\Omega) \rightarrow \V(\T)$ of Chen and Nochetto \cite{CN:00}, which
exhibits optimal approximation properties and preserves positivity, \ie $w\geq 0$ implies $\Pi_\T w \geq 0$. 
This allows us to weaken the regularity assumption on the obstacle to $\psi \in H^1(\Omega)$;
see Remark~\ref{re:obstascle_reg}.
\end{remark}

\begin{remark}[discrete constraints]
Since the discrete unilateral constraint in \eqref{eq:K_discrete} is only enforced at the nodes
of $\T$, its implementation is an easy task in practice. Alternative constraints such as
$W \geq \psi$ or $W \geq \Pi_{\T} \psi$ simplify the error analysis but complicate the implementation.
\end{remark}

\begin{remark}[efficient solvers]
The solution $U_{\T} \in \V(\T)$ of problem \eqref{eq:discrete_obstacle} can be efficiently computed
via monotone multigrid methods; see \cite{MR2867664} and the references therein.
\end{remark}

\subsection{A priori error analysis in the energy norm}
\label{sub:aprioriH1}

We now derive optimal $H^1$-error estimates under the assumption of full regularity. We follow the seminal paper by
Brezzi, Hager and Raviart \cite{BHR:77}; see also \cite{MR0391502}.

\begin{theorem}[$H^1$ error estimate]
\label{TH:H1_opt_est}
Let $f \in L^2(\Omega)$ and $\psi \in H^2(\Omega)\cap C(\bar{\Omega})$ with $\psi \leq 0$ on $\partial\Omega$,
and let $\Omega$ be $C^{1,1}$ or a convex polygon. Then
\begin{equation}
\label{eq:optimal_H1}
\| \nabla(u - U_{\T} ) \|_{L^2(\Omega)} \lesssim h_{\T} ( \| f \|_{L^2(\Omega)} + \| \psi \|_{H^2(\Omega)} ).
\end{equation}
\end{theorem}
\begin{proof}
Notice, first of all, that the underlying assumptions imply that $u \in H^2(\Omega)\cap H^1_0(\Omega)$.
Observe now that 
\begin{equation}\label{eq:aux}
 \| \nabla(u - U_{\T} ) \|^2_{L^2(\Omega)} = 
\left( \nabla(u - U_{\T} ), \nabla(u - I_{\T}u ) \right)
+
\left( \nabla(u - U_{\T} ), \nabla(I_{\T}u -U_{\T}) \right)
\end{equation}
and 
\[
\left( \nabla(u - U_{\T} ), \nabla(u - I_{\T}u ) \right) \leq
\frac{1}{2} \| \nabla(u - U_{\T} ) \|^2_{L^2(\Omega)} + \frac{1}{2} \| \nabla(u - I_{\T}u ) \|^2_{L^2(\Omega)}.
\]
Integrate by parts the second term in \eqref{eq:aux} 
and set $V = I_{\T} u \in \mathcal{K}_{\T}$ in \eqref{eq:discrete_obstacle}, to derive
\[
\left( \nabla(u - U_{\T} ), \nabla(I_{\T}u -U_{\T}) \right) \leq 
\left( -\Delta u - f , I_{\T}u -U_{\T} \right) = \left( \lambda , I_{\T}u -U_{\T} \right).
\]
We rewrite the right hand side of the term above as follows:
\begin{equation}
\label{aux:rhs}
\big( \lambda , I_{\T}u -U_{\T} \big) = \big( \lambda , I_{\T}(u -
\psi) -( u - \psi) \big)
+ \big( \lambda, u -\psi \big) + \big( \lambda, I_{\T} \psi -U_{\T} \big). 
\end{equation}
In view of the complementarity relation in 
\eqref{eq:complementary_system}, we deduce that 
$(\lambda, u - \psi) = 0$. The fact that $U_{\T} \geq I_{\T} \psi$, in conjunction with
$\lambda\ge0$, which also stems from \eqref{eq:complementary_system}, 
implies $(\lambda, I_{\T} \psi - U_{\T} ) \leq 0$.
We can now deal with the first term in \eqref{aux:rhs} via a
standard interpolation estimate
\[
| \left( \lambda , I_{\T}(u-\psi)-(u-\psi) \right) | \lesssim h^2_{\T} \|\lambda \|_{L^2(\Omega)} | u - \psi |_{H^2(\Omega)}.
\]
Combining the preceding estimates with $ \| \nabla(u - I_{\T}u ) \|^2_{L^2(\Omega)} \lesssim 
h_{\T}^2|u|^2_{H^2(\Omega)}$, we obtain
\[
\| \nabla(u - U_{\T} ) \|^2_{L^2(\Omega)} \lesssim h^2_{\T}\left( |u|_{H^2(\Omega)}^2  +  
\|\lambda \|_{L^2(\Omega)} | u - \psi |_{H^2(\Omega)}\right).
\]
This implies \eqref{eq:optimal_H1} and concludes the proof.
\end{proof}

\begin{remark}[optimal error estimate and obstacle regularity]
\label{re:obstascle_reg}
When $n \geq 3$ we can weaken the assumptions in Theorem~\ref{TH:H1_opt_est} from 
$ \psi \in H^2(\Omega) \cap C(\bar{\Omega})$
to $ \psi \in H^2(\Omega)$ by replacing 
$I_{\T}$
with the positivity preserving operator $\Pi_{\T}$ of \cite{CN:00}. 
The key step is the positivity property:
$u - \psi \geq 0$ implies $\Pi_{\T} (u - \psi) \geq 0$ so that $\Pi_{\T} u \in \mathcal{K}_{\T}$.
\end{remark}

\subsection{Pointwise a priori error analysis}
\label{sub:aprioriLinfty}
Let us now obtain pointwise error estimates which, as we will see in \S~\ref{subsec:free_boundary}, are essential to study the approximation of the free boundary $\Gamma$.
This theory relies on two crucial ingredients, whose details are beyond the scope of this paper.
The first one is a quasi-optimal \emph{pointwise error estimate} for the so-called elliptic projection:
\begin{equation}
\label{eq:elliptic_proj}
R_{\T}u \in \V(\T): \qquad \int_{\Omega} \nabla(u - R_{\T}u) \nabla V  = 0 \quad \forall V \in \V(\T). 
\end{equation}

\begin{proposition}[pointwise error estimate]
\label{prop:Linfty}
If $u \in H_0^1(\Omega) \cap W_{\infty}^{2}(\Omega)$, then there exists a constant $C$ independent of $h_{\T}$ and $u$ such that
\begin{equation}
\label{eq:u-Ru}
 \| u - R_{\T}u \|_{L^{\infty}(\Omega)} \leq C h^2_{\T} | \log h_{\T} | |u|_{W_{\infty}^{2}(\Omega)}.
\end{equation}
\end{proposition}
This classical result is due to Nitsche \cite{Nitsche:77}, Scott \cite{Scott:76} and Schatz and Wahlbin \cite{SW:82}. For convenience, in what follows we denote the right hand side of \eqref{eq:u-Ru} by
\begin{equation}
\label{eq:defofeta}
  \eta( h_\T ) := C h_\T^2 |\log h_\T| |u|_{W_{\infty}^{2}(\Omega)}.
\end{equation}

The second ingredient is the \emph{discrete maximum principle} (DMP): The mesh $\T$ is \emph{weakly acute} if
\[
 \int_{\Omega} \nabla \phi_{z} \nabla \phi_{\tilde{z}} \leq 0 \qquad \forall z,\tilde{z} \in \N(\T).
\]
If $n=2$ this property holds over simplicial meshes provided that the sum of angles opposite to a side does not exceed 
$\pi$ \cite[\S2]{CR:73}.

\begin{proposition}[DMP]
\label{pro:dmp}
Let $\T$ be weakly acute and $W \in \V(\T)$ be discrete subharmonic, \ie
\[
 \int_{\Omega} \nabla W \nabla \phi_{z} \leq 0 \qquad \forall z \in \Nin(\T).
\]
Then
$
 \max_{\bar{\Omega} } W = \max_{\partial \Omega} W
$.
\end{proposition}

With the help of the DMP we prove pointwise error estimates
between $u$ and $U_\T$. 
This result is essentially due to Baiocchi \cite{Baiocchi} and Nitsche \cite{Nitsche:77}. 
It extends to more general variational inequalities and lower than $W^2_{\infty}(\Omega)$ regularity.
We begin with a simple but fundamental growth estimate.

\begin{lemma}[growth estimate]
\label{lem:quadgrowth}
Assume that $u,\psi \in W^2_\infty(\Omega)$ and let $T \subset \T$. If there is $x_0 \in T$ such that $(u-\psi)(x_0) = 0$, 
then there is a constant $C$ proportional to $|u-\psi|_{W^2_{\infty}(\Omega)}$
such that
\[
  0 \leq (u-\psi)(x) \leq C h_\T^2 \quad \forall x \in T.
\]
\end{lemma}
\begin{proof}
The fact that $u-\psi \geq 0$ implies $\GRAD(u-\psi)(x_0)=0$. 
Since $u-\psi \in W^2_{\infty}(\Omega)$, we deduce
$\| \GRAD(u-\psi)\|_{L^\infty(T)} \leq C h_\T$.
Applying the mean value theorem yields the asserted estimate.
\end{proof}

\begin{theorem}[pointwise error estimates]
If $\T$ is weakly acute and $u,\psi \in W_{\infty}^{2}(\Omega)$, then
there exists a constant $C^* > 0$ depending on
$|u|_{W_{\infty}^{2}(\Omega)}$ and  $|\psi|_{W_{\infty}^{2}(\Omega)}$
such that, with $\eta(h_\T)$ as in \eqref{eq:defofeta}, we have
 \begin{equation}
\label{eq:Linf_estimate}
 \| u - U_{\T} \|_{L^{\infty}(\Omega)} < C^* \eta( h_{\T} ).
\end{equation}
\end{theorem}
\begin{proof} The proof relies on the construction of discrete super and subsolutions to \eqref{eq:discrete_obstacle}.

\mbox{1}. \emph{Supersolution:} Let $U_{\T}^{+} := R_{\T} u + C_1
\eta(h_{\T})$ with $C_1 > 1$ to be determined. We have
\[
U_{\T}^+ \geq u - \eta(h_{\T}) + C_1 \eta(h_{\T})
 = (u-\psi) + \left( \psi - I_{\T} \psi + (C_1 - 1) \eta(h_{\T}) \right ) + I_{\T} \psi,
\]
where $I_{\T}$ denotes the Lagrange interpolation operator. The fact that $\psi \in W_{\infty}^2(\Omega)$, implies the interpolation estimate
$
 \| \psi - I_{\T} \psi \|_{L^{\infty}(\Omega)} \lesssim h_{\T}^{2} |\psi |_{W_{\infty}^2(\Omega)}
$
which, together with $u \geq \psi$, yields 
$U^{+}_{\T} \geq I_{\T} \psi$ for $C_1$ sufficiently large
depending on $|\psi|_{W_{\infty}^2(\Omega)}$.

For $z \in \Nin(\T)$ set $v = u + \phi_z \in \mathcal{K}$ in \eqref{eq:class_obstacle} to deduce
\[
 \int_{\Omega} \nabla U_{\T}^+ \nabla \phi_{z} \geq \int_{\Omega} f \phi_{z} =: F_{z} \qquad \forall z \in \Nin(\T),
\]
where we used that $(\nabla R_\T u, \nabla \phi_z) = (\nabla u, \nabla \phi_z)$ and that $\nabla U_\T^+ = \nabla R_\T u$.
To prove that $U_{\T}^{+} \geq U_{\T}$, we argue by contradiction. Let 
\[
\N^{+}(\T):= \{ z \in \N(\T): U^{+}_{\T}(z) < U_{\T} (z) \},
\]
and note that $\N^{+}(\T) \subset \Nin(\T)$ because $C_1 > 1$. We thus have $\psi(z) \leq U_{\T}^{+}(z) < U_{\T}(z)$ for all 
$z \in \N^{+}(\T)$, whence
$
 \int_{\Omega} \nabla U_{\T} \nabla \phi_{z} = F_z
$
for all $ z \in \N^{+}(\T)$, according to \eqref{eq:discrete_obstacle}. Consequently 
\[
 \int_{\Omega} \nabla(U_{\T} - U^{+}_{\T}) \nabla \phi_{z} \leq 0 \qquad \forall z \in \N^{+}(\T).
\]
Applying the DMP of Proposition \ref{pro:dmp} to $\N^{+}(\T)$ we infer that
$(U_{\T} - U^{+}_{\T} ) (z) \leq 0$ for all $z \in \N^{+}(\T)$ which contradicts the definition of $\N^{+}(\T)$. Hence
\begin{equation}
\label{aux_inf_1}
 U_{\T} \leq U^{+}_{\T} \leq u + (C_1 + 1 )\eta(h_{\T}),
\end{equation}
where we used the pointwise estimate of Proposition~\ref{prop:Linfty}.

\mbox{2} \emph{Subsolution}. Let $U_{\T}^{-}:= R_{\T} u - C_2 \eta(h_{\T})$ for a suitable constant $C_2 > 1$. 
We again argue by contradiction to prove that $U^{-}_{\T} \leq U_{\T}$. Let
\[
 \N^{-}(\T):= \{ z \in \N(\T): U_{\T}^{-}(z) > U_{\T}(z)\},
\]
and note that $\N^{-}(\T) \subset \Nin(\T)$ because $C_2>1$. 
We have that for each $z \in \N^{-}(\T)$
\[
 \psi(z)\leq U_{\T}(z) < U^{-}_{\T}(z) \leq u(z) - (C_2 - 1 )\eta(h_{\T}).
\]
This implies that $u > \psi$ in $\mathrm{supp}~\phi_{z}$, for otherwise we can apply the quadratic growth estimate of Lemma~\ref{lem:quadgrowth} to obtain
\[
 (C_2 - 1) \eta(h_{\T}) \leq (u - \psi)(z) \leq C \eta(h_{\T}),
\]
which is a contradiction for $C_2$ sufficiently large depending
$|u|_{W_{\infty}^2(\Omega)}$ and $|\psi|_{W_{\infty}^2(\Omega)}$.

Using \eqref{eq:class_obstacle} we get $\int_{\Omega} \nabla u \nabla \phi_z = F_z$
for all $z \in \N^{-}(\T)$. Combining this with \eqref{eq:discrete_obstacle} we derive
\[
 \int_{\Omega} \nabla(U_{\T}^{-} - U_{\T}) \nabla \phi_{z} \leq 0 \qquad z \in \N^{-}(\T),
\]
whence applying Proposition \ref{pro:dmp} (DMP) we deduce $U_{\T}^{-}(z) \leq U_{\T}(z)$ for all $z \in \N^{-}(\T)$. This contradicts the definition
of $\N^{-}(\T)$ and yields
\begin{equation}
\label{aux_inf_2}
 U_{\T} \geq U_{\T}^{-} \geq u - (C_2 + 1) \eta(h_{\T}).
\end{equation}

\mbox{3.} The desired estimate \eqref{eq:Linf_estimate} is finally a consequence of \eqref{aux_inf_1} and \eqref{aux_inf_2}.
\end{proof}

\subsection{Error estimates for free boundaries}
\label{subsec:free_boundary}
With the change of variables $v=u-\psi$ we can assume that the obstacle $\psi = 0$. 
In this setting let us address the approximation of the free boundary $ \Gamma = \partial \Omega^{+} \cap \Omega $. 
We will follow the work of Nochetto \cite{Nochetto:86}, 
which is in turn inspired by the results
of Brezzi and Caffarelli \cite{BC:83}. We present an elementary procedure to 
determine a discrete interface together with sharp convergence rates both in measure and in distance. 
The key ingredients are the so-called \emph{nondegenerancy properties} (NDP) of the obstacle problem 
\eqref{eq:class_obstacle} \cite{Caffarelli:81,Friedman} and the
definition of the discrete free boundary $\Gamma_\T$ and
non-coincidence set $\Omega^+_{\T}$ \cite{Nochetto:86}:
for a parameter $\delta_{\T} > 0$ to be properly selected, we set
\begin{equation}\label{eq:free-bd}
\Gamma_\T := \partial\Omega_\T^+ \cap \Omega,
\qquad
 \Omega^+_{\T}:= \{ x \in \Omega: U_{\T} > \delta_{\T} \}.
\end{equation}

Let us recall the NDP for the obstacle problem \eqref{eq:class_obstacle}. Assume 
$f \in W^1_{\infty}(\Omega)$ and that there is a negative constant $C(f)$ for which
\begin{equation}
\label{eq:f_C(f)}
f \leq C(f) < 0 \textrm{ in } \Omega. 
\end{equation}
Since $f \in W^1_{\infty}(\Omega)$, we have 
$ u \in C^{1,1}(\Omega) \cap W^{s}_p(\Omega)$ for $1 < p < \infty$ and
$s < 2 + \frac{1}{p}$ \cite{Friedman, KS}, and
pointwise regularity $u \in W^2_{\infty}(\Omega)$ up to the fixed
boundary $\partial \Omega$, 
provided $\partial \Omega \in C^{2,\alpha}$ \cite{Friedman, KS}. 

We define 
$
 A(u, \epsilon):= \{ x \in \Omega: 0 < u(x) < \epsilon^2 \}
$,
which clearly depends on how the solution $u$ behaves near the free boundary $\Gamma$, and
$
  \Ss(\Gamma,\epsilon):= \{ x \in \Omega^{+}: \textrm{dist}(x,\Gamma) < \epsilon \}
$.
With this notation at hand,
we present the following NDP properties \cite{Caffarelli:81,Friedman}.

\begin{lemma}[local NDP]
Let $f \in W^1_{\infty}(\Omega)$. If \eqref{eq:f_C(f)} holds, then the set $A(u,\epsilon)$
satisfies
\begin{equation}
\label{NDP_measure}
 |A(u,\epsilon) \cap K| \leq C \epsilon \qquad \forall K \Subset \Omega \qquad \mathrm{(NDP~in~measure)}
\end{equation}
where $C = C(K, \| u \|_{W^2_{\infty}(K)}, \| f \|_{W^1_{\infty}(K)})$.
In addition, we have
\begin{equation}
\label{NDP_distance}
 A(u,\epsilon) \cap K \subset \Ss(\Gamma,\epsilon) \cap K 
\qquad \forall K \Subset \Omega
\qquad \mathrm{(NDP~in~distance)}.
\end{equation}
\end{lemma}

A NDP thus prescribes a precise growth of $u$ away from $\Gamma$ and, as a result, 
$u$ cannot be uniformly flat near $\Gamma$. A local NDP in measure was first established by Caffarelli in 
\cite{Caffarelli:81} under the qualitative property \eqref{eq:f_C(f)} and 
was exploited to derive regularity properties of the free boundary \cite{Caffarelli:81,Friedman}. 
A global NDP in measure can be derived as soon as $u \in W^2_{\infty}(\Omega)$. On the other hand, a NDP in distance
is more intricate in that it entails a pointwise behavior of $u$ near $\Gamma$.
Let us now obtain an error estimate for the free boundaries \cite{Nochetto:86}. 

\begin{theorem}[interface error estimates]
Define $\delta_{\T}:=C^{*}\eta(h_\T)$ with $C^{*}$ 
given in \eqref{eq:Linf_estimate}. Then
\begin{equation}
\label{eq:rate_measure}
  \left| \left(\Omega^{+} \vartriangle  \Omega_{\T}^{+}\right) \cap K
  \right| \leq C \eta(h_{\T})^{1/2}
\qquad \forall K \Subset \Omega.
\end{equation}
If $\partial \Omega \in C^{2,\alpha}$ then one can set $K = \Omega$. 
In addition, we have
\begin{equation}
\label{eq:rate_distance}
 \Gamma_{\T} \cap K \subset \Ss\left( \Gamma,  \sqrt{2\eta(h_{\T})}
 \right) 
\qquad \forall K \Subset \Omega.
\end{equation}
\end{theorem}
\begin{proof}
Since $\Omega^{+} \vartriangle  \Omega_{\T}^{+} = (\Omega^{+} \setminus \Omega_{\T}^{+})  \cup (\Omega_{\T}^{+} \setminus \Omega^{+})$, we analyze each set separately. 
We start by expressing $\Omega^{+} \setminus \Omega_{\T}^{+}$ as follows:
\begin{align*}
\Omega^{+} \setminus \Omega_{\T}^{+} &= \{x \in \Omega: 0 < u(x) < 2 \delta_{\T} ~\mathrm{and}~U_{\T}(x) \leq \delta_{\T}\} 
\\
& \cup \{x \in \Omega: u(x) \geq 2 \delta_{\T} ~\mathrm{and}~U_{\T}(x) \leq \delta_{\T}\} = A \cup B.
\end{align*} 
Since $A \subset \{x \in \Omega: 0 < u(x) < 2 \delta_{\T} \}$, the NDP in measure \eqref{NDP_measure} immediately implies 
$|A\cap K| \leq C(2 \delta_{\T})^{1/2}$ for all $K \Subset \Omega$. To estimate $|B|$, we first note that if $x \in B$ then
$u(x) - U_{\T}(x) \geq \delta_{\T}$. 
Since the pointwise estimate \eqref{eq:Linf_estimate} yields 
$u(x) - U_{\T}(x) < C^* \eta_{\T} = \delta_\T$ for almost every $x$,
we deduce $|B|=0$. We thus arrive at
\[
 \left|\left(\Omega^{+} \setminus \Omega_{\T}^{+}\right) \cap K \right| \leq C \delta_{\T}^{1/2} .
\]
Since 
$
  \Omega_\T^+ \setminus \Omega^+ = \left\{ x \in \Omega: 
u(x) \leq 0 ~\mathrm{and}~ U_\T(x) > \delta_\T \right\},
$
we argue as with $B$ to get $|\Omega_{\T}^{+}\setminus\Omega^{+}|=0$.

We now proceed to obtain \eqref{eq:rate_distance}. If $x \in
\Gamma_{\T}$ then $U_{\T}(x) = \delta_{\T}$, whence estimate \eqref{eq:Linf_estimate}
implies
\[
  0 = U_{\T}(x) - C^*\eta(h_{\T}) < u(x) < U_{\T}(x) + C^* \eta(h_{\T}) 
= 2 \delta_{\T},
\]
so that $x \in A(u,2\delta_\T)$. The inclusion \eqref{eq:rate_distance} now follows from the NDP in distance \eqref{NDP_distance}.
\end{proof}

\section{The thin obstacle problem}
\label{sec:thin_obstacle}
To simplify the discussion we assume that $\Omega$ is convex and replace the differential operators in 
\eqref{eq:class_obstacle} and \eqref{eq:complementary_system} by $-\Delta + I$. The admissible set $\mathcal K$ is replaced by
\[
\mathcal{K}:= \{v \in H^1(\Omega): v \geq g \textrm{ a.e. on } \partial \Omega   \}. 
\]
If $\partial \Omega \in C^{1,1}$ and $g \in C^{1,1/2}(\partial \Omega)$, then
the corresponding variational inequality \eqref{eq:class_obstacle} has a unique solution
$u \in H^2(\Omega) \cap C^{1,1/2}(\bar{\Omega})$; the H\"{o}lder regularity is due to 
Athanasopoulos and Caffarelli \cite{AC:04}. With this regularity $-\Delta u + u =f$ a.e.~$\Omega$
and, if $z = \partial_{\nu} u$ on $\partial \Omega$, the following complementarity system is valid, which
is also known as \emph{Signorini complementarity conditions}:
\begin{equation}
\label{eq:thin_complementary_system}
  z \geq 0, \quad u - g \geq 0, \quad z(u - g) = 0 \quad \textrm{on } \partial \Omega;
\end{equation}
see \cite[Chap. 9]{PSU:12}. Since we do not assume $\Omega$ to be a polyhedral domain, we consider a family $\{ \T  \}$ of triangulations of polyhedral domains $\Omega_{\T}$ that approximate $\Omega$ in such a way that $\N(\T) \subset \bar \Omega_{\T} $, $\N(\T) \cap \partial \Omega_{\T} \subset \partial \Omega$ and $|\Omega\setminus\Omega_\T| \lesssim h_\T^2$. Since $\Omega$ is convex we have that $\Omega_\T \subset \Omega$ and we extend discrete functions $V \in \V(\T)$ to $\bar\Omega \setminus \Omega_{\T}$ by a constant in the direction normal to $\partial \Omega_{\T}$. We define the discrete admissible set by
\begin{equation}
\label{eq:KdiscreteThin}
\mathcal{K}_{\T}:= \{ W \in \V(\T): W \geq I_{\T} g \textrm{ on } \partial \Omega_{\T} \}.
\end{equation}
We present an optimal energy error estimate valid for any dimension and without assumptions on
the free boundary. This improves upon the original result by Brezzi, Hager and Raviart \cite{BHR:77}.
Alternative results in this direction, for $n=2$, can be found in \cite{MR3022224} and references therein.

\begin{theorem}[optimal energy error estimate]
Let $\Omega$ be convex, $C_\Omega>0$ denote the
$C^{1,1}$-seminorm of $\partial\Omega$ and $g \in
C^{1,1/2}(\partial \Omega)$. Then there is $C>0$, that
depends on $|u|_{H^2(\Omega)}$, $|u|_{W^1_\infty(\Omega)}$,
$|u|_{C^{1,1/2}(\bar\Omega)}$, $|g|_{C^{1,1/2}(\partial \Omega)}$,
and $C_\Omega$, for which
\begin{equation}
\label{eq:thin_estimate}
\| u - U_{\T} \|_{H^1(\Omega)} \leq C h_{\T}.
\end{equation}
\end{theorem}
\begin{proof}
The $C^{1,1/2}$ regularity of $u$ implies that $0 \leq z = \partial_\nu u \in C^{0,1/2}(\partial\Omega)$. We now
proceed as in Theorem \ref{TH:H1_opt_est} and write
\begin{equation*}
  \begin{aligned}
    \| u - U_\T \|_{H^1(\Omega)}^2 &= \int_{\Omega} \nabla (u - U_{\T}) \nabla (I_{\T}u - U_{\T})+ (u-U_{\T}) (I_{\T}u - U_{\T}) \\
    &+ \int_{\Omega} \nabla (u - U_{\T}) \nabla ( u - I_{\T}u) +   (u-U_{\T}) (u - I_{\T}u) = \textrm{I} + \textrm{II}.
  \end{aligned}
\label{auxthin}
\end{equation*}
To estimate \textrm{I} we exploit that $u \in H^2(\Omega)$ and integrate by parts to obtain
\begin{align*}
\textrm{I}
\leq \int_{\Omega}(-\Delta u + u - f)(I_{\T}u - U_{\T})
+ \int_{\partial \Omega} z(I_{\T}u - U_{\T}),
\end{align*}
where we have also used that $U_\T$ solves the discrete problem. 
Since $-\Delta u + u = f$ a.e. in $\Omega$, the first term
  vanishes. For the second term, we write
\[
\int_{\partial \Omega} z (I_{\T}u - U_{\T}) = 
\int_{\partial \Omega} z (I_{\T}(u-g) - (u-g))
+
\int_{\partial \Omega} z  (u-g)
+
\int_{\partial \Omega} z  (I_{\T}g - U_{\T}).
\]
The complementarity conditions \eqref{eq:thin_complementary_system} imply
$
 \int_{\partial \Omega} z(u-g) = 0.
$
Since $U_\T \in \mathcal{K}_\T$ we have that $U_{\T} \geq I_{\T} g$ on $\partial \Omega$ which yields
$
 \int_{\partial \Omega} z(I_{\T}g-U_{\T}) \leq 0.
$
For the remaining term we only need to consider faces $S$ on $\partial \Omega_{\T}$ for which, on the corresponding
subtended hypersurfaces $\hat{S}$, $u- g$ is not identically zero nor strictly positive for otherwise
$I_{\T}(u-g) = u-g =0$ or $z = 0$, respectively. 
If $S$ is one of these faces, then 
there exists $x_0 \in \hat{S}$ such that the tangential gradient
$
 \nabla_{\Gamma}(u-g)(x_0) = 0
$
because $u-g \geq 0$.
Since $u-g \in C^{1,1/2}(\partial\Omega)$, employing an argument
similar to Lemma~\ref{lem:quadgrowth}, we find $C>0$ that
depends only on $|u-g|_{C^{1,1/2}(\partial\Omega)}$ for which
\[
 0 \leq (u-g)(x) \leq Ch_{\T}^{3/2} \qquad \forall x \in \hat{S}.
\]
Conditions \eqref{eq:thin_complementary_system} imply there is $x_1 \in \hat{S}$ with $z(x_1)=0$. Now, since $u \in C^{1,1/2}(\bar\Omega)$, we have that
$
 0 \leq z(x) \leq Ch_{\T}^{1/2}
$
for all $x \in \hat{S}$,
with $C = |u|_{C^{1,1/2}(\bar\Omega)}$.
Combine these estimates to obtain
\[
\textrm{I} \leq \int_{\partial \Omega} z(I_{\T}(u-g) - (u-g)) \leq Ch_{\T}^2 | \partial \Omega|.
\]
We now split the term \textrm{II} into two contributions
  \textrm{II}$(\Omega_\T)$ and
  \textrm{II}$(\Omega\setminus \Omega_{\T}$) over the sets $\Omega_\T$ and 
$\Omega\setminus \Omega_{\T}$ respectively. 
For the first one, we use interpolation theory to get
$
\|  u - I_{\T}u \|_{H^1(\Omega_{\T})} \leq C h_{\T} |u|_{H^2(\Omega)}.  
$
We estimate the remaining term
\textrm{II}$(\Omega\setminus \Omega_{\T})$ as follows:
\begin{equation*}
  \textrm{II}(\Omega\setminus\Omega_\T) 
 \leq \frac12 \| u - U_\T \|_{H^1(\Omega\setminus \Omega_\T)}^2 + 
 C|u|_{W^1_\infty(\Omega)}^2 |\Omega \setminus \Omega_\T|,
\end{equation*}
where we used that $I_\T$ is stable in
$W^1_\infty(\Omega_\T)$. Since $\partial\Omega$ is $C^{1,1}$,
we realize that
\[
 \textrm{II} \leq \frac12\| u - U_\T \|_{H^1(\Omega)}^2 + C
\left( |u|_{H^2(\Omega)}^2 + C_\Omega |u|_{W^1_\infty(\Omega)}^2 \right) h_{\T}^2 ,
\]
which, together with the estimate for $\textrm{I}$, yields the assertion \eqref{eq:thin_estimate}.
\end{proof}

\section{The fractional obstacle problem}
\label{sec:fractional_obstacle}
We now consider the obstacle problem for fractional powers of the
Laplace operator, which is a rather recent development. 
To be concrete, let $\Omega$ be an open, bounded, convex and connected domain
of $\R^n$ ($n\ge1$), with polyhedral boundary $\partial\Omega$, 
and for $s\in (0,1)$ let $\Hs$ be
\begin{equation}
  \label{H}
  \Hs = H^s(\Omega) \quad s \in (0,\frac12), \quad
  \mathbb{H}^{1/2}(\Omega) = H_{00}^{1/2}(\Omega), \quad
  \Hs = H_0^s(\Omega) \quad s \in (\frac12,1),
\end{equation} 
which is a Hilbert space, and $\Hsd$ be its dual. Given
$f \in \Hsd$ and an obstacle 
$\psi \in \Hs \cap C(\bar{\Omega})$ satisfying $\psi \leq 0$ on
$\partial \Omega$, the fractional obstacle problem reads
\begin{equation}
\label{eq:frac_obstacle}
u \in \mathcal{K}:\quad \langle \Laps u, u - w \rangle_{\Hsd \times \Hs} \leq 
\langle f, u - w \rangle_{\Hsd \times \Hs}
\quad \forall w \in \mathcal{K},
 \end{equation}
where $\mathcal{K} := \{ w\in\Hs: ~ w\ge \psi \text{ a.e.~in } \Omega \}$
is the convex, closed, and nonempty set of admissible displacements.
Among the several definitions of $\Laps$, we adopt that in \cite{NOS} 
which is based on the spectral theory of the Dirichlet Laplacian.

The problem \eqref{eq:frac_obstacle} has a unique solution 
$u \in \Hs$ \cite{Friedman,KS}. 
Moreover, if we assume that $\Omega$ and $f$ are such that $\Laps u \in L^2(\Omega)$,
we can rewrite \eqref{eq:frac_obstacle} as a complementarity system:
\begin{equation}
\label{eq:frac_complementary_system}
\Laps u - f \geq 0, \quad u - \psi \geq 0, \quad (\Laps u - f)(u - \psi) = 0 \quad \textrm{ a.e. on } \partial \Omega.
 \end{equation}

\subsection{Localization and Truncation}\label{sub:localize}
The optimal H\"older regularity of the solution $u$ to
\eqref{eq:frac_complementary_system}, namely
$
u \in C^{1,s}(\Omega),
$
has been studied 
by Caffarelli, Salsa and Silvestre
\cite{CSS:08}
using the extension proposed by Caffarelli and Silvestre in \cite{CS:07} to $\mathbb{R}^{n+1}_+$, which we now review.
The problem
\begin{equation}
\label{eq:linearLaps}
  \Laps u = f \ \text{in } \Omega, \qquad u = 0 \ \text{on } \partial\Omega,
\end{equation}
is equivalent to the Neumann-to-Dirichlet map associated with the following nonuniformly elliptic
mixed boundary value problem
\begin{equation}
\label{extension} 
  \DIV\left( y^\alpha \nabla \ue \right) = 0 \ \text{in } \C, \quad
  \ue = 0 \  \text{on } \partial_L \C, \quad
  \partial^{\alpha}_{\nu} \ue = d_s f \ \text{on } \Omega \times \{0\},
\end{equation}
where $\alpha := 1-2s$, $\C := \Omega \times (0,\infty) = \{ x = (x',y) \in \R_{+}^{n+1}: x' \in \Omega, y > 0 \}$,
$\partial_L \C  := \partial \Omega \times [0,\infty)$ corresponds to the lateral boundary
of $\C$, 
$\partial^{\alpha}_{\nu} \ue  = -\lim_{y \rightarrow 0^+} y^\alpha \ue_y,$
and $d_s$ denotes a constant that depends only on $s$; see \cite{CS:07,NOS} for details.
Upon defining the weighted Sobolev space
$
  \HL(y^{\alpha},\C) := \left\{ w \in H^1(y^\alpha,\C): w = 0 \textrm{ on } \partial_L \C\right\}
$,
the linear problem \eqref{extension} can be formulated weakly as follows:
\[
\ue \in \HL(y^{\alpha},\C):  \quad \int_{\C} y^{\alpha} \nabla \ue \nabla w = d_s 
\langle f, \tr w \rangle_{\Hsd \times \Hs} \quad \forall w \in \HL(y^{\alpha},\C).
\]
For $w \in \HL(y^{\alpha},\C)$, we denote by $\tr w$ its trace onto
$\Omega \times \{ 0 \}$, and we recall that the trace operator $\tr$ satisfies (see \cite[Proposition 2.5]{NOS})
\begin{equation}
\label{Trace_estimate}
\tr \HL(y^\alpha,\C) = \Hs,
\qquad
  \|\tr w\|_{\Hs} \leq C_{\tr} \| w \|_{\HLn(y^\alpha,\C)}.
\end{equation}

This extension problem is due to Caffarelli and Silvestre \cite{CS:07} for $\Omega = \R^n$;
see \cite{NOS} and references therein for its modification to bounded domains. With its aid, we recast \eqref{eq:frac_obstacle} as follows
\begin{equation}
\label{eq:extended_obstacle}
\ue \in \mathcal{K}: \quad 
\int_{\C} y^{\alpha} \nabla \ue \nabla (\ue-w) \leq 
d_s \langle f, \tr (\ue-w)\rangle_{\Hsd \times \Hs} \quad \forall w \in \mathcal{K},
\end{equation}
where
$
 \mathcal{K}:= \{ w \in \HL(y^{\alpha},\C): \tr w \geq \psi \textrm{ a.e. in } \Omega   \}
$
denotes the set of admissible displacements.
If $\ue$ solves \eqref{eq:extended_obstacle}, then $u = \tr \ue$
solves \eqref{eq:frac_obstacle}. The obstacle
constraint is thus applied on part of $\partial\C$.

We now exploit the exponential decay of the 
solution $\ue$ to \eqref{eq:extended_obstacle} to truncate the cylinder $\C$.

\begin{lemma}[exponential decay]
\label{lem:exdecay}
Let $\ue \in \HL(y^\alpha, \C)$ denote the solution to
\eqref{eq:extended_obstacle}. 
Then, we have
\[
\| \GRAD \ue \|_{L^2(y^\alpha,\Omega\times(\Y,\infty))} \lesssim e^{-\sqrt{\lambda_1} \Y/2} \left( \| \psi \|_{\Hs} + \| f \|_{\Hsd} \right) 
  \quad \forall \Y \ge 1.
\]
\end{lemma}
\begin{proof}
Notice that $\ue \in \HL(y^\alpha, \C)$ solves
\[
  \DIV\left( y^\alpha \nabla w \right) = 0 \ \text{in } \C, \quad
  w = 0 \  \text{on } \partial_L \C, \quad
  \tr w = \tr \ue \ \text{on } \Omega \times \{0\}.
\]
The representation formula provided in \cite[(2.24)]{NOS}, together with the decay estimates of \cite[Proposition 3.1]{NOS}, applied to this problem yield
\[
  \| \GRAD \ue \|_{L^2(y^\alpha,\Omega\times(\Y,\infty))} \lesssim e^{-\sqrt{\lambda_1} \Y/2} \| \tr \ue \|_{\Hs}.
\]
Stability of problem \eqref{eq:extended_obstacle} in terms of
  $\psi$ and $f$ allows us to deduce the desired estimate.
\end{proof}
In view of such an exponential decay, we truncate the cylinder $\C$ to a
height $y=\Y$, i.e. set $\C_{\Y}:= \Omega \times (0,\Y)$; see \cite[Theorem 3.5]{NOS}
for the linear case. We then consider
\begin{equation}
\label{eq:truncobstacle}
  \ve \in \mathcal{K}_{\Y}: \quad 
  \int_{\C_\Y} y^{\alpha} \nabla \ve \nabla (\ve-w) \leq 
  d_s \langle f, \tr (\ve-w)\rangle_{\Hsd \times \Hs} \quad \forall
  w \in \mathcal{K}_{\Y},
\end{equation}
where
$
 \mathcal{K}_{\Y}:= \{ w \in\ \HL(y^\alpha, \C_\Y): \tr w \geq \psi  \textrm{ a.e. in } \Omega \}.
$
As in \cite{NOS}, we only incur in an exponentially small error 
by considering the truncated version 
\eqref{eq:truncobstacle} of problem \eqref{eq:extended_obstacle}.
In addition, $\ve$ satisfies a complementarity system such as 
\eqref{eq:frac_complementary_system}:
\begin{equation}\label{eq:frac_complementary_extension}
\ze:= \partial^\alpha_\nu\ve - d_s f\ge 0,
\quad
\tr\ve - \psi \ge 0, 
\quad
\ze \big(\tr \ve - \psi \big) = 0
\quad \text{a.e. in } \Omega.
\end{equation}

Using the extended problem \eqref{eq:extended_obstacle} to
$\mathbb{R}^{n+1}_+$, Caffarelli, Salsa and Silvestre 
\cite[Theorem 6.7]{CSS:08} showed that if $x_0' \in \Omega$ and
$\ue(x_0',0)-\psi(x_0') = 0$ then
\begin{equation}
\label{reg:CSS}
 0 \leq  \ue(x',0) - \psi(x') \leq C |x'-x_0'|^{1+s}
\end{equation}
for a positive constant $C$. This implies that $u=\ue(\cdot,0) \in
C^{1,s}(\Omega)$ and $\partial_\nu^\alpha \ue(\cdot,0) \in C^{0,1-s}(\Omega)$
because $\partial_\nu^\alpha \ue(\cdot,0) = d_s (-\Delta)^s u$. 
Since their techniques are
\emph{local}, they apply to our extension to the cylinder $\C$ and
truncation to $\C_\Y$ as well, and give the
optimal H\"older regularity for $\ve$:
\begin{equation}\label{eq:optimal_reg}
\partial^\alpha_\nu \ve(\cdot,0) \in  C^{0,1-s}(\Omega).
\end{equation}

These results account only for the regularity of $\ve$ in $\Omega$.
  However, we also need to know the regularity of $\ve$
  over the cylinder $\C_\Y$, which is established by Allen et. al.
  \cite[Theorem 6.4]{ALP:14}:
\begin{equation}
\label{eq:optimal_reg_cyl}
  s \leq \frac12 \ \Rightarrow \ve \in C^{0,2s}(\C_\Y); \qquad s > \frac12 \ \Rightarrow \ve \in C^{1,2s-1}(\C_\Y).
\end{equation}

\subsection{Discretization}\label{sub:discfracobs}
In \cite{NOS} we have utilized the local approach of Caffarelli and Silvestre 
\cite{CS:07} to approximate problem \eqref{eq:linearLaps}.
After having truncated $\C$ to $\C_{\Y}$, the next issue is to compensate
for the rather singular behavior of $\ve(\cdot,y)$ by \emph{anisotropic} meshes
$\mathcal{I}_\Y$ of $[0,\Y]$ with mesh points:
\begin{equation}
\label{graded_mesh}
  y_k = \left( \frac{k}{M}\right)^{\gamma} \Y, \quad k=0,\dots,M, \quad \gamma > 3/(1-\alpha)=3/2s > 1.
\end{equation}
This entails the development
of a polynomial interpolation theory in weighted Sobolev spaces over anisotropic meshes
with Muckenhoupt weights; see \cite{NOS,NOS2}. 

Let $\T_\Omega = \{K\}$ be a conforming and shape regular mesh of
$\Omega$ as in \S~\ref{subsec:fe_approximation}.
The collection of these triangulations $\T_{\Omega}$ is denoted by $\Tr_\Omega$. 
We construct the mesh $\T_{\Y}$ of the cylinder $\C_{\Y}$
as the tensor product triangulation 
of $\T_\Omega$ and $\mathcal{I}_\Y$, and denote the set
of all such triangulations $\T_{\Y}$ by $\Tr$.
We assume that there is a constant $\sigma_{\Y}$ such that if $T_1=K_1\times I_1$ and $T_2=K_2\times I_2 \in \T_\Y$ have nonempty intersection, then
$ h_{I_1}^{-1} h_{I_2} \leq \sigma_{\Y}$,
where $h_I = |I|$. This weak regularity condition on the mesh allows for 
anisotropy in the extended variable \cite{DL:05,NOS,NOS2}. 
For $\T_{\Y} \in \Tr$, we define
the finite element space 
\begin{equation}
\label{eq:FESpace}
  \V(\T_\Y) = \left\{
            W \in C^0( \overline{\C_\Y}): W|_T \in \mathcal{P}(K) \otimes \mathbb{P}_1(I) \ \forall T \in \T_\Y, \
            W|_{\Gamma_D} = 0
          \right\},
\end{equation}
where $\Gamma_D = \partial_L \C_{\Y} \cup \Omega \times \{ \Y\}$ is the
Dirichlet boundary and $\mathcal{P}(K)$ is defined as in \S~\ref{subsec:fe_approximation}.
The Galerkin approximation of the truncated problem is then given by
\begin{equation}
\label{harmonic_extension_weak}
 U_{\T_{\Y}} \in \V(\T_{\Y}): \
  \int_{\C_\Y} y^{\alpha}\nabla U_{\T_{\Y}} \nabla W = d_s \langle f, \textrm{tr}_{\Omega} W \rangle_{\Hsd \times \Hs}
  \quad \forall W \in \V(\T_{\Y}).
\end{equation}

Notice that $\#\T_{\Y} = M \, \# \T_\Omega$ and $\# \T_\Omega \approx M^n$
imply $\#\T_\Y \approx M^{n+1}$. Finally, if $\T_\Omega$ is shape regular and quasi-uniform, we have
$h_{\T_{\Omega}} \approx (\# \T_{\Omega})^{-1/n}$. All these considerations allow us to obtain the
following result; see \cite[Theorem 5.4]{NOS} and \cite[Corollary 7.11]{NOS}.

\begin{theorem}[a priori error estimate]
\label{TH:fl_error_estimates}
Let $\T_\Y \in \Tr$ be a tensor product grid, which is quasiuniform in $\Omega$ and graded in the 
extended variable so that \eqref{graded_mesh} holds. If $\V(\T_\Y)$ is defined by \eqref{eq:FESpace}, $U_{\T_\Y} \in \V(\T_\Y)$ solves \eqref{harmonic_extension_weak} 
and $\ue \in \HL(y^\alpha,\C)$ solves \eqref{extension}, then we have
\begin{equation*}
\label{optimal_rate}
  \| \ue - U_{\T_\Y} \|_{\HLn(y^\alpha,\C)} \lesssim
|\log(\# \T_{\Y})|^s(\# \T_{\Y})^{-1/(n+1)} \|f \|_{\mathbb{H}^{1-s}(\Omega)},
\end{equation*}
where $\Y \approx \log(\# \T_{\Y})$. Alternatively, if $u$ solves \eqref{eq:linearLaps}, then
\begin{equation*}
\| u - U_{\T_\Y}(\cdot,0) \|_{\Hs} \lesssim
|\log(\# \T_{\Y})|^s(\# \T_{\Y})^{-1/(n+1)} \|f \|_{\mathbb{H}^{1-s}(\Omega)}.
\end{equation*}
\end{theorem}

The discretization of 
\eqref{eq:frac_obstacle} is then carried out by a 
Galerkin approximation to \eqref{eq:truncobstacle},
namely
\begin{equation}
\label{eq:FEfracobstacle}
  V_{\T_\Y} \in \mathcal{K}_{\T_\Y}: \
  \int_{\C_\Y} y^{\alpha} \nabla V_{\T_\Y} \nabla (V_{\T_\Y}-W) \leq 
  d_s \langle f, \tr (V_{\T_\Y}-W)\rangle_{\Hsd \times \Hs},
\end{equation}
for all $W \in \mathcal{K}_{\T_\Y}$. The discrete admissible set $\mathcal{K}_{\T_{\Y}}$ is defined by
\begin{equation}
\label{eq:K_discreteFlap}
\mathcal{K}_{\T_{\Y}}:= \{ W \in \V(\T_{\Y}): \tr W \geq \tr \Pi_{\T_{\Y}} \Psi \textrm{ on } \Omega \},
\end{equation}
where $\Psi \in \HL(y^\alpha, \C_\Y)$ is the $\alpha$-harmonic extension of $\psi$ and $\Pi_{\T_{\Y}}:L^1(\C_\Y) \to \V(\T_\Y)$ is the quasi-interpolation operator on the mesh $\T_{\Y}$ defined in \cite{NOS,NOS2}. This operator is constructed by local averaging,
is able to separate the variables $x'\in\Omega$ and $y\in(0,\Y)$, is stable in $\HL(y^\alpha,\C_\Y)$ and $L^p(\C_\Y)$ and exhibits {\it quasi-local} optimal approximation
properties for any $1\le p \le \infty$.

\subsection{Energy error analysis}\label{sub:apriorifracobs}

We build on the preceding approach to obtain an almost optimal error estimate for the solution of \eqref{eq:frac_obstacle}. 
We start by quantifying the error $\ue-\ve$ due to truncation of $\C$.

\begin{proposition}[exponential error estimate]
\label{prop:experr}
Let $\ue \in \mathcal{K}$  and $\ve \in \mathcal{K}_{\Y}$ solve problems \eqref{eq:extended_obstacle}
and \eqref{eq:truncobstacle}, respectively. If $\Y\ge1$, then we have
\[
  \| \GRAD( \ue - \ve ) \|_{L^2(y^\alpha,\C_\Y)} \lesssim e^{-\sqrt{\lambda_1} \Y/8} \left( \| \psi \|_{\Hs} + \| f \|_{\Hsd} \right).
\]
\end{proposition}
\begin{proof}
By definition, we have
\[
  \|\GRAD (\ue - \ve) \|_{L^2(y^\alpha,\C)}^2 = \int_\C y^\alpha \GRAD
  \ue \GRAD (\ue - \ve ) - \int_\C y^\alpha \GRAD \ve \GRAD(\ue - \ve).
\]
We examine each term separately. For the first term, since $\ve
\in \mathcal{K}$, we set $w = \ve$ in \eqref{eq:extended_obstacle} to get
\begin{equation}
\label{eq:KK}
  \int_\C y^\alpha \GRAD \ue \GRAD (\ue - \ve) \leq d_s \langle f, \tr(\ue - \ve) \rangle_{\Hsd \times \Hs}.
\end{equation}
For the second term we would like to set $w = \ue$ in \eqref{eq:truncobstacle}, but
this is not a valid test function because $\ue \notin \HL(y^\alpha,\C_\Y)$.
Instead, let $\rho \in W^1_\infty(0,\infty)$ be a smooth cutoff function such that $\rho \equiv 1$ on $\left[0,\tfrac\Y2\right]$, 
$\rho = \tfrac{2}{\Y}(\Y-y)$ on $\left[\tfrac\Y2,\Y\right]$ and $\rho = 0$ on $[\Y,\infty)$; see \cite[(3.7)]{NOS}.
Then, we write
\begin{equation}
\label{aux}
  \int_\C y^\alpha \GRAD \ve \GRAD(\ve-\ue) = \int_{\C_\Y} y^\alpha \GRAD \ve \GRAD(\ve - \rho \ue) + \int_{\C_\Y} y^\alpha \GRAD \ve \GRAD\left( (\rho-1)\ue\right).
\end{equation}
Set $w = \rho \ue \in \mathcal{K}_\Y$ in \eqref{eq:truncobstacle} to obtain
\begin{equation}
\label{eq:PP}
  \int_{\C_\Y} y^\alpha \GRAD \ve \GRAD(\ve - \rho \ue) \leq d_s \langle f, \tr(\ve-\ue) \rangle_{\Hsd \times \Hs}.
\end{equation}
In light of \eqref{eq:KK}--\eqref{eq:PP}, we see that the result will follow if we bound the last term in \eqref{aux}. 
For that we use Lemma~\ref{lem:exdecay} and the arguments of \cite[Lemma 3.3]{NOS} to obtain
\begin{align*}
  \int_{\C_\Y} y^\alpha \GRAD \ve \GRAD\left( (\rho-1)\ue\right)  \lesssim e^{-\sqrt{\lambda_1} \Y/4}( \| \psi \|_{\Hs} + \| f \|_{\Hsd} )
  \|\GRAD \ve \|_{L^2(y^\alpha,\C_\Y)},
\end{align*}
which readily yields the asserted estimate.
\end{proof}

We now present an almost optimal a priori error estimate which
relies on \cite{BHR:77} and Theorem~\ref{TH:H1_opt_est}.

\begin{theorem}[almost optimal error estimate]
Let $\T_\Y \in \Tr$ be the tensor product grid described in Theorem \ref{TH:fl_error_estimates}.
If $\ue \in \mathcal{K}$ and $V_{\T_{\Y}} \in \mathcal{K}_{\T_{\Y}}$ solve \eqref{eq:extended_obstacle}
and \eqref{eq:FEfracobstacle}, respectively, we have
\begin{equation}
\label{obs_optimal_rate}
  \| \ue - V_{\T_\Y} \|_{\HLn(y^\alpha,\C)} \leq C
|\log(\# \T_{\Y})|^s(\# \T_{\Y})^{-1/(n+1)},
\end{equation}
Alternatively, if $u$ solves \eqref{eq:frac_obstacle}, then
\begin{equation}
\label{frac_suboptimal_rate}
\| u - V_{\T_\Y}(\cdot,0) \|_{\Hs} 
\leq C
|\log(\# \T_{\Y})|^s(\# \T_{\Y})^{-1/(n+1)}.
\end{equation}
where, in both inequalities, $C$ depends on the H\"older moduli of 
smoothness of $\ve$ given by \eqref{eq:optimal_reg} and \eqref{eq:optimal_reg_cyl}, $\|f\|_{\Hsd}$ and $\| \psi \|_{\Hs}$.
\end{theorem}
\begin{proof}
We first compare the solution $\ve$ of the truncated problem and the discrete solution $V_{\T_\Y}$. To do so, we proceed as in Theorem~\ref{TH:H1_opt_est}. We just need to examine
\begin{multline*}
\int_{\C_{\Y}} y^{\alpha} \nabla(\ve - V_{\T_{\Y}}) \nabla(\Pi_{\T_{\Y}}\ve - V_{\T_{\Y}})
 \leq  - \int_{\C_{\Y}} \textrm{div}(y^{\alpha} \nabla \ve) \left(\Pi_{\T_{\Y}}\ve - V_{\T_{\Y}} \right)
\\ + \int_{\Omega \times \{ 0 \} } \left( \partial_{\nu}^{\alpha} \ve -d_s f \right) 
\tr \left( \Pi_{\T_{\Y}}\ve - V_{\T_{\Y}} \right)
 = \int_{\Omega \times \{ 0 \} } \ze \tr ( \Pi_{\T_{\Y}}\ve - V_{\T_{\Y}} ),
\end{multline*}
because $\textrm{div}(y^{\alpha} \nabla \ve)=0$ in $\C_\Y$ and $\ze$ is defined in \eqref{eq:frac_complementary_extension}.
We can rewrite the preceding term as
\begin{equation}
\label{eq:bdryint}
\begin{aligned}
  \int_{\Omega \times \{ 0 \} } \ze  \tr ( \Pi_{\T_{\Y}} \ve - V_{\T_{\Y}} )
& = \int_{\Omega \times \{ 0 \} } \ze  \tr ( \ve - \Psi ) 
  + \int_{\Omega \times \{ 0 \} } \ze \tr ( \Pi_{\T_{\Y}} \Psi - V_{\T_{\Y}} ) \\
&+ \int_{\Omega \times \{ 0 \} } \ze \left( \tr \Pi_{\T_{\Y}}(\ve-\Psi) - \tr(\ve-\Psi) \right) .
\end{aligned}
\end{equation}
Using \eqref{eq:frac_complementary_extension}, together with $\tr\Psi = \psi$ and 
$\tr \big(V_{\T_{\Y}} - \Pi_{\T_{\Y}} \Psi\big) \ge 0$, we infer the estimates
$\int_{\Omega \times \{ 0 \} } \ze  \tr(\ve - \Psi)= 0$ and 
$\int_{ \Omega \times \{ 0 \} } \ze  \tr ( \Pi_{\T_{\Y}} \Psi - V_{ \T_{\Y} } )  \leq 0$.
We now express the last term on the right hand side of \eqref{eq:bdryint} as follows:
\[
  \calI := 
         \sum_{K \in \T_\Omega} \int_K \ze \left( \tr \Pi_{\T_{\Y}}(\ve-\Psi) - \tr(\ve-\Psi) \right)
        = \sum_{K \in \T_\Omega} \calI(K).
\]
We next examine separately the cells $K \in \T_\Omega$
according to the value of $\ve - \psi$ in $S_K$, a discrete
neighborghood of $K$.
The issue at stake is that the interpolation operator
$\Pi_{\T_\Y}$ hinges on
local averages over a discrete neighborghood
$S_T$ of $T= K \times I$. This leads to the following three cases.

Case \mbox{1}: $\ve - \psi > 0$ in $S_K$. In this situation, $\ze
\equiv 0$ in $S_K$ and thus $\calI(K)$ vanishes.

Case \mbox{2}: $\ve - \psi \equiv 0$ in $S_K$. The fact that
  $\Pi_{\T_\Y}$ is max-norm stable locally yields
\begin{equation}
\label{eq:ninja}
  \| \tr \Pi_{\T_\Y} (\ve - \Psi) \|_{L^\infty(K)} \lesssim \| \ve -
  \Psi \|_{L^\infty(S_T)} \lesssim h_I^{2s}. 
\end{equation}
The case $s \leq \tfrac12$ follows immediately from \eqref{eq:optimal_reg_cyl}. For $s>\tfrac12$, we use that
$\mathscr{E} :=\ve - \Psi$ solves an $\alpha$-harmonic extension problem and then its conormal derivative 
$\partial_\nu^\alpha \mathscr{E} = -\lim_{y\to0} y^{1-2s}\partial_y\mathscr{E}$ is well
defined. We realize that $\partial_y\mathscr{E}|_{y=0}$ vanishes in $\Omega$ because
$s>\frac12$, which together with \eqref{eq:optimal_reg_cyl} allows us to derive \eqref{eq:ninja}.
This, combined with \eqref{eq:optimal_reg}, implies the following bound
\[
  \calI(K) \leq \| \ze \|_{L^\infty(K)} \| \tr \Pi_{\T_\Y} (\ve - \Psi) \|_{L^\infty(K)} |K| \lesssim |K| h_I^{2s}.
\]
Since $h_I \approx (\# \T_\Y)^{-\gamma/(n+1)}$ and $\gamma > 3/2s$, we
thus conclude $\calI(K) \lesssim |K| (\# \T_\Y)^{-2/(n+1)}$.

Case \mbox{3}: $\tr\ve - \psi$ is not identically zero nor strictly
positive in $S_K$. In view of \eqref{reg:CSS}, we have
\begin{equation}
\label{eq:growthx}
 0 \le (\tr \ve - \psi)(x')
 \lesssim h_{\T_{\Omega}}^{1+s} \approx (\# \T_{\Y} )^{-(1+s)/(n+1)} 
\quad \forall x' \in S_K.
\end{equation}
On the other hand, using \eqref{eq:optimal_reg} we deduce 
$\ze = \partial_{\nu}^{\alpha} \ve - d_s f \in C^{1-s}(\bar\Omega)$ and
\begin{equation}
\label{eq:zgrowth}
0 \leq \ze(x') \lesssim (\# \T_{\Y} )^{-(1-s)/(n+1)}
\quad \forall x' \in S_K.
\end{equation}
Consequently,
  $\int_K \ze \tr (\ve -\Psi ) \lesssim |K| (\# \T_\Y)^{-2/(n+1)}$.
The fact that there is a point $x_0' \in S_K$ where $(\tr\ve-\psi)(x_0')=0$, 
and an argument similar to the one that led us to
\eqref{eq:ninja} on the basis of \eqref{eq:optimal_reg_cyl}, in conjunction with
\eqref{eq:growthx} and \eqref{eq:zgrowth}, yield
\[
  \| \tr \Pi_{\T_\Y} (\ve - \Psi) \|_{L^\infty(K)} \lesssim \| \ve - \Psi \|_{L^\infty(S_T)} 
  \lesssim \max\big\{h_{\T_\Omega}^{1+s}, h_I^{2s} \big\} \lesssim h_{\T_\Omega}^{1+s}
\]
and $\int_K \ze \tr \Pi_{\T_\Y}
  (\ve -\Psi ) \lesssim |K| (\# \T_\Y)^{-2/(n+1)}$. Hence
$
\calI(K) \lesssim |K| (\# \T_\Y)^{-2/(n+1)}.
$

Collecting the estimates for the three cases, we thus conclude
\[
   \| \ve - V_{\T_\Y} \|_{\HLn(y^\alpha,\C_\Y)} \lesssim \Y^s (\# \T_{\Y})^{-1/(n+1)},
\]
where $\Y$ accounts for the interpolation
estimate of $\|\ve-\Pi_{\T_\Y}\ve\|_{\HLn(y^\alpha,\C_\Y)}$
based on the mesh grading \eqref{graded_mesh} \cite{NOS}.
The estimate \eqref{obs_optimal_rate} follows from
Proposition~\ref{prop:experr} and a suitable choice of the parameter
$\Y$ in terms of $\# \T_{\Y}$; see \cite[Remark 5.5]{NOS}. 
Finally, \eqref{Trace_estimate} and \eqref{obs_optimal_rate} lead to
\eqref{frac_suboptimal_rate}.
\end{proof}

\section{Conclusions and open problems}\label{sec:conclusion}

Several topics of interest were only mentioned in passing or not at all. Among them mixed methods, 
monotone multigrid methods and a posteriori error estimation come to mind. Other discretization techniques 
such as nonconforming finite elements, virtual elements, or mimetic finite differences have not been described.
Let us also list some open problems of interest. 
The convergence properties of multigrid methods in higher dimensions is still not well understood.
The lack of duality techniques makes obtaining error estimates 
in $L^2(\Omega)$ for the classical obstacle problem rather difficult. Pointwise error estimates for the 
thin and fractional obstacle problems are rather technical but
appear to be an interesting open issue.
Other obstacle-type problems, such as time dependent problems 
and high-order equations, might require techniques other
than those presented here and were not discussed.

\bibliographystyle{plain}
\bibliography{biblio}

\end{document}